\documentclass[times]{amsart}
\usepackage{amssymb, amsmath, amsthm, amsfonts}

\newtheorem{thm}{Theorem}[section]
\newtheorem{cor}[thm]{Corollary}
\newtheorem{lem}[thm]{Lemma}
\newtheorem{prop}[thm]{Proposition}

 \newtheorem*{hypWS}{Hypothesis S}
\newtheorem*{MPCC}{Montgomery's Pair Correlation Conjecture}

\newcommand{\s}{\sigma}
\renewcommand{\b}{\beta}
\renewcommand{\r}{\rho}
\newcommand{\g}{\gamma}
\newcommand{\z}{\zeta}

\renewcommand{\d}{\delta}
\newcommand{\e}{\epsilon}
\renewcommand{\l}{\log}
\renewcommand{\L}{\Lambda}

\renewcommand{\Re}{{\mathfrak{R}\,}}

\def\({\left(}
\def\){\right)}
\def\intl{\int\limits}

\newcommand{\ga}{g}
\newcommand{\gad}{g^*}

\newcommand\be{\begin{equation}}
\newcommand\ee{\end{equation}}
\newcommand\bea{\begin{eqnarray}}
\newcommand\eea{\end{eqnarray}}
\newcommand\bi{\begin{itemize}}
\newcommand\ei{\end{itemize}}
\newcommand\ben{\begin{enumerate}}
\newcommand\een{\end{enumerate}}
\newcommand\bes{\begin{equation*}}
\newcommand\ees{\end{equation*}}

\def\({\left(}
\def\){\right)}

\def\C {{\mathbb C}}

\def\R {{\mathbb R}}
\def\Z {{\mathbb Z}}
\def\N {{\mathbb N}}
\def\intl{\int\limits}

\begin{document}

\title{$a$-Points of the Riemann zeta-function on the critical line}

\author{Stephen J. Lester}

\address{Department of Mathematics, University of Rochester, Rochester, NY 14627 USA}

\curraddr{School of Mathematical Sciences, Tel Aviv University, Tel Aviv, 69978 Israel}

\email{slester@post.tau.ac.il}

\subjclass[2010]{11M06, 11M26, 60F05.}

\thanks{The author was supported in part by the NSF grant DMS-1200582.}

\begin{abstract}
We investigate the proportion of the nontrivial 
roots of the equation $\z(s)=a$, which lie on the line $\Re s=1/2$ for $a \in \C$
not equal to zero. We show that at most one-half of these points lie on the line $\Re s=1/2$.
Moreover, assuming a spacing condition on the ordinates of zeros of the Riemann zeta-function,
we prove that
zero percent of the nontrivial solutions to $\z(s)=a$ lie on the line $\Re s=1/2$ for any
nonzero complex number $a$. 
\end{abstract}

\maketitle

\section{Introduction}

Let $s=\s+it$ be a complex variable, $\z(s)$ be the Riemann zeta-function, 
and $a$ be a nonzero complex number. The solutions
to $\z(s)=a$, which we will denote by $\r_a=\b_a+i\g_a$, are called $a$-points,
and their distribution has been widely studied. 
For principal references see \cite{T:1986}, \cite{L:1975}, and \cite{S:1992}.

For every $a$ there
is a $n_0(a)$ so that for all $n \geq n_0$ there is an $a$-point of $\z(s)$ quite close to $s=-2n$.
Moreover, in the half-plane $\s \leq 0$ there are only finitely many other $a$-points; and
we call the $a$-points with real part $\le 0$ trivial $a$-points. The remaining $a$-points all lie
in a strip $0< \s <A$, where $A$ depends on $a$, and are called nontrivial $a$-points. 
Let $\r_a=\b_a+i\g_a$ denote a nontrivial $a$-point. The number of these
is given by
\be \label{count apts}
N_a(T)=\sum_{1 < \g_a \leq T} 1=\frac{T}{2\pi}\l \frac{T}{2\pi}-\frac{T}{2\pi}+O(\l T),
\ee
for $a \neq 1$ (this holds for $a=0$ as well). In the case $a=1$ there is an additional $-(\log 2) T/(2\pi)$ term
on the right-hand side of the equation (see Levinson \cite{L:1975}).

By analogy with the case $a=0$, it is natural to investigate the distribution of the
nontrivial $a$-points. Let
$$
N_a(\s_1, \s_2; T)=\sum_{\substack{0 < \g_a \leq T \\  \s_1< \b_a <\s_2}} 1.
$$
For fixed $1/2<\s_1 < \s_2<1$ and $a \neq 0$ Borchsenius and Jessen \cite{B:1948} showed 
there exists a constant
$C(a,\s_1,\s_2)>0$ such that
\be \label{strip approx}
N_a(\s_1, \s_2; T) = C(a, \s_1, \s_2)T+o(T) \qquad \qquad (T \rightarrow \infty).
\ee
As for the case $a=0$, it is well-known that there can be at most $\ll T^{\theta(\s_1)}$,
with $\theta(\s_1)<1$, in such a strip (for instance, see Chapter IX of \cite{T:1986}). 

Levinson \cite{L:1975} studied $a$-points near the critical line and showed that for any $\d>0$
$$
\sum_{\substack{ 0< \g_a \leq T \\ 1/2-\d < \b_a <1/2+\d }} 1=\frac{T}{2\pi} \l T +O_{\d}(T),
$$
where the implied constant depends on $\d$. From \eqref{count apts},
it immediately follows that $ N_a(T)(1+o(1))$ $a$-points of $\z(s)$ lie in the strip $1/2-\d<\s<1/2+\d$, 
$1< t< T$, 
for any fixed $\d$. Therefore, the $a$-points of $\z(s)$ cluster near the critical line $\s=1/2$.

Selberg also studied the distribution of the $a$-points of $\z(s)$ near the critical line.
Under the assumption of 
the Riemann hypothesis, Selberg, in unpublished work, 
showed for $c>0$ and $\s=1/2-c\sqrt{\pi \l \l T}/ \l T$ that 
$$
\sum_{\substack{1 < \g_a \leq T \\ \b_a> \s}} 1 = \int_{-c}^{\infty} e^{-\pi x^2} dx \, \, N_a(T)(1+o(1))
$$
(a proof may be found in Tsang's thesis \cite{Ts:1984}).
By this we see that, assuming the Riemann hypothesis,  
about $1/2$ of the $a$-points of $\z(s)$ lie 
to the left of the line $\s=1/2$ at distances
of order $\sqrt{ \l \l T}/ \l T$. Taking $c \rightarrow 0^+$
slowly, it follows that
\be \label{selberg result}
N_a(0, 1/2; T) \geq \tfrac12 \cdot N_a(T)(1+o(1)).
\ee

Understanding the distribution of the remaining one-half of the $a$-points seems to be
quite difficult. Selberg \cite{S:1992} states that most of these points lie quite
close to the critical line at distances 
of order not exceeding $(\l \l \l T)^3/(\l T \sqrt{\l \l T})$
away from the critical line. 
Moreover, 
he conjectured that approximately one half of these lie 
to the left of the line $\s=1/2$ while the other half
lie to the right. 
That is, three-quarters of the nontrivial $a$-points lie to the left
of the critical line $\s=1/2$, while the remaining one-quarter lie to the right of the line.

On the critical line $\s=1/2$ we expect that there are very few $a$-points. In fact,
Selberg \cite{S:1992} has conjectured
that there are at most only finitely many $a$-points on
the critical line.
Observe that for almost all $a$ there are no $a$-points
of $\z(s)$ with $\b_a=1/2$. This is because the set 
$\{ \z(\tfrac12+it) : t \in \R \} \subset \C$, 
has two-dimensional Lebesgue measure zero. 
Recently, Banks \emph{et. al} \cite{Ba:2012} 
have shown that the curve $\{ \z(\tfrac12+it): t\in \R\}$
has countably many self-intersections. From this, it immediately follows that there are
only countably many numbers $a$ for which more than one $a$-point lies on the critical line. 
On the other hand, for every complex number $a$ we have from \eqref{selberg result} that 
the Riemann hypothesis implies that no more
than one-half of the $a$-points can lie on the line $\s=1/2$.

We will investigate the number of $a$-points that lie on the line $\s=1/2$ 
both unconditionally
and under the assumption of a spacing condition on 
the ordinates of zeros of the Riemann zeta-function. Unconditionally,
we show that for any nonzero complex number $a$ at most one-half of the
nontrivial $a$-points
of the Riemann zeta-function lie on the critical line. Furthermore, under the 
assumption of
a spacing condition
 we prove
that almost all of the $a$-points of $\z(s)$ do not lie on the line $\s=1/2$.

\section{Main Results}

Let us first introduce Hardy's $Z$-function
in the form
$$
Z(t)=e^{i\theta(t)}\z(\tfrac12+it),
$$
where $\theta(t)$ is the Riemann-Siegel theta function and is given by
$$
\theta(t)=\arg \Gamma(\tfrac14+i\tfrac{t}{2})-\frac{\l \pi}{2}t.
$$
Next, for any nonzero $a \in \C$ we write
$a=|a|e^{i\phi}$ with $-\pi < \phi \leq \pi$ and 
let $g=g(\phi)$ be a solution to
$$
\theta(t) \equiv -\phi  \pmod \pi.
$$ 
We call the points, $g$, \emph{shifted Gram points} and for any $n \in \Z$
we call the unique $g=g_n$ that satisfies $\theta(g_n)=\pi n-\phi$
the \emph{$nth$ shifted Gram point}.

Let
$$
\Psi=\tfrac12  \l \l T.
$$
Also, write $\mathbf 1_{S}(x)$ for the indicator function of the set $S$; that is $\mathbf 1_S(x)$ equals
one if $x \in S$ and equals zero otherwise.
A result of Selberg states for $\alpha < \beta$ that
$$
\frac{1}{T} \int_T^{2T} \mathbf 1_{[\alpha, \beta]} \Big( \l |\z(\tfrac12+it)|\, \Psi^{-1/2} \Big)dt = \frac{1}{\sqrt{2\pi}}
\int_{\alpha}^{\beta} e^{-x^2/2} dx +O\Big(\frac{(\l \Psi)^2}{\sqrt {\Psi }}\Big)
$$
(see \cite{S:1992} and \cite{Ts:1984}). Since shifted Gram points are regularly spaced (see Lemma \ref{simple lem}) 
it seems reasonable to expect that $\log |\zeta(\tfrac12+ig)| \Psi^{-1/2}$ is normally distributed
at the shifted Gram points $T< g \le 2T$. In fact,
Selberg, in unpublished work, proved that $\arg \zeta(\tfrac12+ig) \Psi^{-1/2}$ has
a normal limiting distribution.

However, there are extra considerations that need to be accounted for
when estimating the real part of the logarithm.  
For instance, the possible existence of Landau-Siegel zeros could cause the ordinates of zeros of $\zeta(s)$
to be distributed according to the Alternative Hypothesis as stated
in \cite{F:2013}. Hence, the presence
of Landau-Siegel zeros could imply that
a positive proportion of the shifted Gram points 
$T< g \le 2T$ are equal to, or at least extremely close to, ordinates of zeros
of $\zeta(s)$. This would show that $\log |\zeta(\tfrac12+ig)| \Psi^{-1/2}$ is not normally distributed
at these points.

For this reason we assume
\begin{hypWS}\label{hyp: WS}
For any $n \in \mathbb N$
$$
\lim_{\e \rightarrow 0^+} \limsup_{T \rightarrow \infty}
\frac{|\{T < \g, \g' \leq 2T : |\frac{(\g-\g') \l T}{2 \pi}-n|< \e \}|}{T \l T}=0,
$$
where $\g,\g'$ are ordinates of zeros of the Riemann zeta-function.
\end{hypWS}
We note that this spacing hypothesis is similar to other spacing
hypotheses made in \cite{BH2:1987}, \cite{H:1987}, and \cite{BH:1995}.
Additionally, we observe that  if Montgomery's Pair Correlation Conjecture is true
then so is Hypothesis~S.
\begin{MPCC}
For fixed $\alpha < \beta$ as $T \rightarrow \infty$
$$
\frac{1}{N(T)} 
\sum_{\substack{0 < \g, \g' \leq T  \\ \frac{ 2\pi\alpha}{\l T} \leq \g-\g' \leq \frac{ 2\pi\beta}{\l T}}} 1  \sim 
\int_{\alpha}^{\beta} 1-\Big(\frac{\sin \pi x}{\pi x}\Big)^2 \, dx
+\mathbf 1_{[\alpha, \beta]}(0).
$$
\end{MPCC}

To calculate the distribution function of $\log |\zeta(\tfrac12+ig)| \Psi^{-1/2}$  at
the shifted Gram points $T< g \le 2T$ we will assume the truth of Hypothesis~S.
However, using the method of B. Hough ~\cite{Hou:2011} we can {\em unconditionally} establish an upper bound for
the distribution function.

\begin{thm} 
\label{upper bd}
For fixed $\alpha \in \mathbb R$ we have as $T \rightarrow \infty$
\bes
\frac{1}{\frac{T}{2\pi} \log T} \sum_{\substack{T < g \leq 2T \\ g \neq \g}}
\mathbf 1_{[\alpha, \infty)} \Big( \l |\z(\tfrac12+ig)|\, \Psi^{-1/2} \Big) \le \frac{1}{\sqrt{2\pi}}
\int_{\alpha}^{\infty} e^{-x^2/2} dx +o(1).
\ees
\end{thm}

Assuming Hypothesis~S we show that the upper bound
is the best possible.
\begin{thm} 
\label{thm 2}
Assume Hypothesis~S. For fixed $\alpha < \beta$
$$
 \lim_{T \rightarrow \infty} \frac{1}{\frac{T}{2\pi} \l T} 
 \sum_{\substack{T < g \leq 2T \\ g\neq \g}}
\mathbf 1_{[\alpha, \beta]} \Big( \l |\z(\tfrac12+ig)|\, \Psi^{-1/2} \Big) = \frac{1}{\sqrt{2\pi}}
\int_{\alpha}^{\beta} e^{-x^2/2} dx .
$$
\end{thm}
Establishing an upper bound 
on the rate of convergence to the Gaussian distribution in Theorem \ref{thm 2} would immediately lead to an 
improvement in Corollary \ref{thm 1}. However, we are unable to do so.
The limitation in our argument arises solely from Hypothesis~S.

Recall that it follows from the work of Selberg (see \cite{S:1992} and \cite{Ts:1984}) that
assuming the Riemann hypothesis at most one-half of the nontrivial $a$-points
can lie on the critical line $\sigma=1/2$.
From Theorem \ref{upper bd} we have that this holds unconditionally.
\begin{cor} \label{half}
For every nonzero complex number $a$ we have
\[
\frac{1}{N_a(T)} \sum_{\substack{  0<\g_a \leq T \\  \b_a=1/2 }} 1
\leq \frac12+o(1)   \qquad \qquad \qquad (T \rightarrow \infty).
\]
That is, at most one-half of the nontrivial $a$-points of the Riemann zeta-function lie on the critical line.
\end{cor}

Similarly, by Theorem \ref{thm 2} we have
the following
\begin{cor} \label{thm 1} Assume Hypothesis~S. For every nonzero $a\in \C$, 
zero percent of the nontrivial $a$-points of $\z(s)$ lie on the critical line.
That is,
$$
\lim_{T \rightarrow \infty} \frac{1}{N_a(T)} \sum_{\substack{  0<\g_a \leq T \\  \b_a=1/2 }} 1 =0.
$$
\end{cor}

By computing a mollified second moment of $\zeta(\tfrac12+ig)-a$ one may be able to give an alternative proof that a positive proportion of the $a$-points do not lie on the line $\s=1/2$.  However, due to the constraint on the 
length of the mollifier we believe that this would give
an inferior result to Corollary \ref{half}. 

As previously mentioned, an additional assumption on the zeros of $\zeta(s)$ is necessary
for the conclusion of Theorem \ref{thm 2} to hold. 
We wonder if
it is possible for
 the zeros of $\zeta(s)$ to be distributed in such a way so that a positive proportion of the 
nontrivial $a$-points lie on the line $\s=1/2$, for some $a \neq 0$.
For instance, if the Alternative Hypothesis as stated in \cite{F:2013} is true then
 does the conclusion of Corollary \ref{thm 1} still hold? It is possible to give examples of functions
$f(s)$ that are analytic in $0< \sigma< 1$, whose zeros are regularly spaced, lie on the line $\s=1/2$, and
for which many of the solutions
to $f(s)=a$ also lie on the line $\s=1/2$. Simple examples
of such functions are
$f(s)=a \sinh(s-\tfrac12)$ and $f(s)=\frac{|a|}{2} \chi(s)+\frac{a}{2}$ with $t>10$ (here $\chi$ is the functional equation factor).
It would be interesting to determine what other functions also have these properties.


We now prove the corollaries.

\begin{proof}[Proof of Corollary \ref{half}] 
For
any nonzero $a \in \C$ write
$a=|a|e^{i\phi}$ with $-\pi<\phi\leq \pi$.
Note that by the functional equation for $\z(s)$ it follows that $Z(t)$ is real.
Thus, at an $a$-point of the form $\r_a=1/2+i\g_a$, 
we have $Z(\g_a)=e^{i\theta(\g_a)}a=e^{i(\theta(\g_a)+\phi)}|a|$. Since
$Z(\g_a)$ is real we have that
$$
\theta(\g_a) \equiv -\phi  \pmod \pi,
$$ 
which implies that $\g_a$ is a shifted Gram point. Hence,
\begin{equation} \label{apoint bd 1}
\sum_{\substack{ T < \g_a \leq 2T \\ \b_a=1/2}} 1 \leq \sum_{\substack{T < g \leq 2T \\ |\z(\frac12+ig)|=|a|}} 1.
\end{equation}

Next, let $A= |\l |a||$ and note that for any $\e >0$, if $T$ is sufficiently large, then $2 A  < \e \sqrt{\Psi}$. Thus,
\begin{equation} \label{apoint bd 2}
\begin{split}
 \sum_{\substack{T < g \leq 2T \\ |\z(\frac12+ig)|=|a|}} 1  \leq & 
 \sum_{\substack{T < g \leq 2T \\ g \neq \g}}\mathbf 1_{[-2A, 2 A]} \Big( \l |\z(\tfrac12+ig)| \Big) \\
 \leq & 
 \sum_{\substack{T < g \leq 2T \\ g \neq \g}}\mathbf 1_{[-\e, \e]} \Big( \l |\z(\tfrac12+ig)|\Psi^{-1/2} \Big) \\
\leq & \sum_{\substack{T < g \leq 2T \\ g \neq \g}}\mathbf 1_{[-\e, \infty)} \Big( \l |\z(\tfrac12+ig)|\Psi^{-1/2} \Big) . 
 \end{split}
\end{equation}
Observe that by Theorem \ref{upper bd} we have
\[ 
\frac{1}{\frac{T}{2\pi} \log T}\sum_{\substack{T < g \leq 2T \\ g \neq \g}}\mathbf 1_{[-\e, \infty)} \Big( \l |\z(\tfrac12+ig)|\Psi^{-1/2} \Big) \leq \frac{1}{\sqrt{2\pi}}
\int_{-\varepsilon}^{\infty} e^{-x^2/2} dx +o(1)=\frac12+o(1)
\]
since $\e>0$ is arbitrary.
Combining this with \eqref{apoint bd 1} and \eqref{apoint bd 2}, we have
$$
\sum_{\substack{ T < \g_a \leq 2T \\ \b_a=1/2}} 1 \le \Big(\frac12+o(1)\Big) \frac{T}{2\pi}\log T.
$$

For any positive integer $N$ observe that we have
$$
\sum_{\substack{\frac{T}{2^N} < \g_a \leq T \\ \b_a=1/2 }} 1 
=\sum_{k=0}^{N-1} \sum_{\substack{\frac{T}{2^{k+1}} < \g_a \leq \frac{T}{2^k} \\ \b_a=1/2 }} 1 
\le \Big(\frac12 +o(1)\Big)\bigg(\sum_{k=0}^{N-1}\frac{T}{2^{k+1}(2\pi)} \l \frac{T}{2^{k+1}} \bigg)=
\Big(\frac12+o(1)\Big)\frac{T}{2\pi}\log T,
$$
since we may take $N$ arbitrarily large.
Recall that
$$
N_a(T)=\sum_{1< \g_a \le T} 1=\frac{T}{2\pi} \l \frac{T}{2\pi}+O(T).
$$
From this we see that
$$
\sum_{\substack{ 0 < \g_a \leq \frac{T}{2^N} \\ \b_a=1/2 }} 1 \leq N_a(T/2^N) \ll \frac{T}{2^N} \l T=o(T \l T).
$$
Corollary \ref{half} now follows.
\end{proof}
\begin{proof}[Proof of Corollary \ref{thm 1}] 
The argument is similar to the previous proof.
Assuming Hypothesis~S we have by Theorem \ref{thm 2} for any fixed
$\e>0$ that as $T \rightarrow \infty$
\bes
\begin{split}
\sum_{\substack{T < g \leq 2T \\ g \neq \g}}\mathbf 1_{[-\e, \e]} \Big( \l |\z(\tfrac12+ig)|\Psi^{-1/2} \Big) =
\frac{T \l T}{(2\pi)^{3/2}} \Bigg( \int_{-\epsilon}^{\epsilon} e^{-x^2/2} \, dx+o(1)\Bigg) 
=o(T \l T)
\end{split}
\ees
since $\e$ was arbitrary.
Using this estimate in  \eqref{apoint bd 2} we have by \eqref{apoint bd 1} that
\[
\sum_{\substack{ T < \g_a \leq 2T \\ \b_a=1/2}} 1=o(T \log T).
\]
The proof is completed along the same lines as before.
\end{proof}

\section{Preliminary Lemmas}
The following estimate is due to van der Corput and can 
be found in \cite{T:1986}.
\begin{lem} \label{vander}
Suppose that $f(u)$ is real and twice differentiable and that $f''(u) \approx \lambda$ on an 
interval $[a,b]$ with $b \geq a+1$. Then
$$
\sum_{a < n \leq b} \exp(2 \pi i f(n)) \ll (b-a)\lambda^{1/2}+\lambda^{-1/2}.
$$
\end{lem}

\begin{lem} \label{oscillation}
Let $T \geq 10$. Then for positive $x$ not equal to one
$$
\sum_{T < g \leq 2T} x^{ig} \ll \(T \frac{ |\l x|}{\l T}\)^{1/2}+\(T \frac{ \l^3 T}{|\l x|}\)^{1/2}+|\l x|.
$$
\end{lem} 
\begin{proof}
Our first step will be to derive an approximate formula for $g_{n}$. By
Stirling's formula, for $t>1$,
\be \label{stirling formula}
\frac{\theta(t)}{\pi}+\frac{\phi}{\pi}=\frac{t}{2\pi} \l \frac{t}{2\pi e}-\frac18+\frac{\phi}{\pi}+O(1/t).
\ee
We now let $\tilde g_{n}$ be defined via the equation
$$
\frac{\tilde g_{n}}{2\pi} \l \frac{\tilde g_{n}}{2\pi e}-\frac18+\frac{\phi}{\pi}=n.
$$
Next write $Z=\l (\tilde g_{n}/(2\pi e))$ so that
$$
\frac{n}{e}+\frac{1}{8e}-\frac{\phi}{\pi e}=Ze^Z.  
$$
Writing, $W(Z)$ for the Lambert $W$-function, which is the inverse function of $Ze^Z$, we have
$W\(\frac{n}{e}+\frac{1}{8e}-\frac{\phi}{\pi e}\)=Z=\l (\tilde g_{n}/(2\pi e))$, so that
$$
\tilde g_{n}=2\pi \exp\(1+W\(\frac{n}{e}+\frac{1}{8e}-\frac{\phi}{\pi e}\) \).
$$

Now suppose that $n > (\theta(T)+\phi)/ \pi$, so that
$g_{n}>T$. By definition,
$$
\(\frac{\tilde g_{ n}}{2\pi} \l \frac{\tilde g_{ n}}{2\pi e}-\frac18+\frac{\phi}{\pi}\)
-\(\frac{\theta(g_{ n})}{\pi}+\frac{\phi}{\pi} \)=n-n=0
$$
and by \eqref{stirling formula},
$$
\frac{\theta(g_{ n})}{\pi}=\frac{ g_{ n}}{2\pi} \l \frac{ g_{ n}}{2\pi e}-\frac{1}{8}+O(1/T).
$$
Consequently,
$$
\Big|\frac{ g_{ n}}{2\pi} \l \frac{ g_{ n}}{2\pi e}-\frac{\tilde g_{ n}}{2\pi} \l \frac{\tilde g_{ n}}{2\pi e}\Big| \ll 1/ T.
$$
Next consider the function $f(x)=x \l (x/(2\pi e))$. For any $x>y \geq 10$
$$
\int_y^x f'(t) \, dt \geq (x-y)\l \frac{y}{2\pi}.
$$
Hence, if $f(x)-f(y) \leq X$ then $x-y\leq X/\l (y/ 2 \pi )$. Thus,
for $n>(\theta(T)+\phi)/\pi$ we have $g_{ n}=\tilde g_{ n}+O(1/(T \l T))$. 
Let $A(t)=(\theta(t)+\phi)/\pi$ and note that
for real $a,b$ we have $|\exp(ib)-\exp(ia)|=|\int_a^b \exp(it) dt |\leq |b-a|$. Therefore,
\be \label{sum cancellation}
\sum_{T < g \leq 2T} x^{ig}=\sum_{A(T)< n \leq A(2T)} x^{ig_n}
=\sum_{A(T)< n \leq A(2T)} x^{i\tilde g_n}+O(|\l x|).
\ee

Let $F(u)=\l x \exp\(1+W\(\frac{u}{e}+\frac{1}{8e}-\frac{\phi}{\pi e}\) \)$
and observe that $\exp(2\pi i F(n))=x^{i\tilde g_n}$. Note that
$$
W'(u)=\frac{1}{\exp(W(u))(1+W(u))}=\frac{W(u)}{u(1+W(u))},
$$ 
and $W(u)\approx \l u$.  So that for $u \geq 10$
$$
F''(u)=\frac{-\l x \, W\(\frac{u}{e}+\frac{1}{8e}-\frac{\phi}{\pi e}\)}
{\(u+\frac{1}{8}-\frac{\phi}{\pi}\)\(1+W\(\frac{u}{e}+\frac{1}{8e}-\frac{\phi}{\pi e}\)\)^3}
\approx \frac{|\l x|}{u \l^2 u}.
$$
We take $a= A(T)$, $b=A(2T)$, $f(u)=F(u)$,
and $\lambda=|\l x|/(T \l^3 T)$. The result now follows from applying Lemma \ref{vander} to
the sum on the right-hand side of \eqref{sum cancellation}.
\end{proof}

Let
$$
N_g(T)=\sum_{0 < g \leq T} 1=\frac{T}{2\pi} \l \frac{T}{2\pi e}+O(1) \qquad \mbox{and}
\qquad N_g(T,2T)=N_g(2T)-N_g(T).
$$
\begin{lem}
Let $2 \leq x \leq T^{1/4}$ and for each prime $p$ let $a_p$ be a complex number. If
$$
a_p \ll \frac{\l p}{p^{1/2} \l x},
$$
then
\be \label{MV 1}
\sum_{T < g \leq 2T}  \bigg| \sum_{p \leq x} a_p p^{-ig} \bigg|^2 \ll T \l T.
\ee
Also, if
$$
a_p \ll 1,
$$
then
\be \label{MV 2}
\sum_{T < g \leq 2T}  \bigg| \sum_{p \leq x} \frac{a_p}{p^{1+2ig}} \bigg|^2 \ll T \l T.
\ee
\end{lem}
\begin{proof}
We begin by proving the first assertion. By Lemma \ref{oscillation},
\be \label{lem 1}
\begin{split}
\sum_{T < \ga \leq 2T} 
 \bigg| \sum_{p \leq x} a_p p^{-ig} \bigg|^2=&N_{\ga}(T, 2T) \sum_{p \leq x} |a_p|^2
+\sum_{\substack{ p, q \leq x \\ p\neq q}} a_p \bar{a}_q \sum_{T < \ga \leq 2T} \(\frac{p}{q}\)^{i\ga} \\ 
\ll & T \frac{\l T}{\l^2 x}\sum_{p \leq x} \frac{\l^2 p}{p}+T^{1/2} \frac{(\l T)^{3/2}}{\l^2 x} 
\sum_{\substack{ p, q \leq x \\ p\neq q}} \frac{\l p \l q}{\sqrt{pq| \l \frac{p}{q}|}}.
\end{split}
\ee
Since
$$
\sum_{p \leq x} \frac{\l^2 p}{p} \ll \l ^2 x,
$$
the first term above is $\ll T \l T$. To estimate the second term on the right-hand side,
observe that if $p< q \leq x$ 
$$
\l q-\l p = \int_p^q \frac{dt}{t} \geq \frac{q-p}{p} \geq \frac1x.
$$
Hence,
$$
\sum_{\substack{ p, q \leq x \\ p\neq q}} \frac{\l p \l q}{\sqrt{pq| \l \frac{p}{q}|}}
\ll x^{1/2} \bigg(\sum_{ p \leq x } \frac{\l p}{\sqrt{p}}\bigg)^2 \ll x^{3/2}.
$$
So that
$$
T^{1/2} \frac{(\l T)^{3/2}}{\l^2 x} 
\sum_{\substack{ p, q \leq x \\ p\neq q}} \frac{\l p \l q}{\sqrt{pq| \l \frac{p}{q}|}} 
\ll T^{1/2} x^{3/2} \frac{(\l T)^{3/2}}{\l^2 x} \ll T.
$$
Applying this in \eqref{lem 1} yields the first assertion of the lemma.

As for the second assertion we argue similarly to obtain
\bes 
\begin{split}
\sum_{T < \ga \leq 2T} 
\bigg| \sum_{p \leq x} \frac{a_p}{p^{1+2i\ga}}\bigg|^2
\ll  T \l T \sum_{p \leq x} \frac{1}{p^2}+T^{1/2} (\l T)^{3/2}
\sum_{\substack{ p, q \leq x \\ p\neq q}} \frac{1}{pq \sqrt{| \l \frac{p}{q}|}},
\end{split}
\ees
which, as before, is seen to be $\ll T \l T$.
\end{proof}
Next, we have
\begin{lem} \label{simple lem} 
For any $m, \ell \in \N$ satisfying $N_g(T) < m < \ell \leq N_g( 2T)$, we have
$$
(g_{\ell}-g_m)\frac{\l T}{2 \pi}=(\ell-m)(1+O(1/\l T)).
$$ 
\end{lem}
\begin{proof}
By Stirling's formula, for $N_g(T) < \ell \leq N_g( 2T)$
$$
\ell=\frac{\theta(g_{\ell})}{\pi}+\frac{\phi}{\pi}=\frac{g_{\ell}}{2\pi}(\l T+O(1)).
$$
Thus,
$$
(\ell-m)=(g_{\ell}-g_m)\frac{\l T}{2\pi}(1+O(1/\l T)).
$$
\end{proof}

The next lemma is from K. M. Tsang's PhD thesis \cite{Ts:1984} and follows from the
zero density estimate 
$$
\sum_{\substack{0 < \g \leq T \\ \b> \s }} 1 \ll T^{1-(\s-1/2)/4} \l T
$$
due to Selberg \cite{Se:1946}.
\begin{lem} \label{Tsang zero}
Let $3 \leq \xi \leq T^{1/8}$. For $k\geq 0$ and $1/2 \leq \s \leq 1$, we have
$$
\sum_{\substack{  0 < \g \leq T \\  \b >\s}} (\b-\s)^k \mathbf \xi^{\b-\s} \ll 
T^{1-(\s-1/2)/4} (\l T)^{1-k},
$$
where the implied constant depends on $k$.
\end{lem}
For $X >0$ and $t \geq 2$ we define the number 
\be \label{sxt def}
\s_{X, t}=\frac12+2\max\Big(\b-\frac12, \frac{2}{\l X} \Big),
\ee
where the maximum is taken over $\r$ satisfying $|t-\g| \leq X^{3|\b-1/2|}/\l X$.
\begin{lem} \label{zero density bd}
Let $3 \leq \xi \leq T^{1/25}$ and $ X= T^{1/100}$. Then for $k\geq0$ we have
$$
\sum_{T < \ga \leq 2T} (\s_{X, \ga}-\tfrac12)^k \xi^{ \s_{X, \ga}-1/2} \ll T \l^{1-k} T,
$$
where the implied constant depends on $k$.
\end{lem}
\begin{proof}
By the definition of $\s_{X,t}$, if for some $g$ we have
$\s_{X,g}>1/2+4/\l X$ then there is a $\r_0$ such that
$\b_0>1/2+2/\l X$ and $|g-\g_0| \leq X^{3(\b-1/2)}/\l X$. In particular, if  $T < g \leq 2T$,
then $0< \g_0 \leq 3T$. Let
\[
G(\b+i\g)=\Big|\Big\{T <\ga \leq 2T : |\ga-\g| \leq \frac{X^{3(\b-1/2)}}{\l X} \Big\}\Big|.
\]
Then
\be \label{sums bounding}
\begin{split}
\sum_{T < \ga \leq 2T}  (\s_{X, \ga}-\tfrac12)^k \xi^{ \s_{X, \ga}-1/2} \ll &
\sum_{T < \ga \leq 2T}  \frac{1}{\l^k X} 
+ \sum_{\substack{ 0< \g \leq 3T \\ \b > 1/2+\frac{2}{\l X} }} (\b-\tfrac12)^k \xi^{ 2(\b-1/2)} 
G(\b+i\g).
\end{split}
\ee
For any zero of $\z(s)$, we have that
\[
G(\b+i\g) \ll \frac{X^{3(\b-1/2)}\l T}{\l X},
\]
because the points g are regularly spaced approximately $(\l T)^{-1}$ apart.
Hence,
the right-hand side of \eqref{sums bounding} is
$$
\ll T \frac{\l T}{\l^k X} +\frac{\l T}{\l X}
\sum_{\substack{  0 < \g \leq 3T \\ \b > 1/2}} (\b-\tfrac12)^k (\xi^2 X^3)^{ \b-1/2}.
$$
Applying Lemma \ref{Tsang zero}, we see that both terms are
$$
\ll T \l^{1-k} T.
$$
\end{proof}

\begin{lem} 
\label{moments lem}
Let  
$m \in \mathbb{N}$ and $2 \leq Y \leq T^{1/m}$. Then
\bes
\frac1T\int_0^{T}  \bigg(\sum_{p \leq Y} \frac{\cos(t\l p)}{p^{1/2}}  \bigg)^m \, dt= 
\int\limits_{[0, 1]^{\pi(Y)}} \!  \bigg(\sum_{p \leq Y} \frac{\cos(2 \pi \theta_p)}{p^{1/2}} \bigg)^m \, d\theta  
+O\big(T^{-1/2}(cm)^{m/2} \big),
\ees
where $d\theta=\prod_{p \leq Y} d\theta_p$ and $c$ is an absolute constant. If $m=0$ this holds 
without the error term.
Furthermore, for $m \in \N$ we have 
\bes
\frac1T\int_0^{T}  \bigg|\sum_{p \leq Y} \frac{\cos(t\l p)}{p^{1/2}}  \bigg|^m \, dt,  \qquad
\int\limits_{[0, 1]^{\pi(Y)}} \!  \bigg|\sum_{p \leq Y} \frac{\cos(2 \pi \theta_p)}{p^{1/2}} \bigg|^m \, d\theta
\ll (\Psi cm)^{m/2},
\ees
where $c$ is a positive absolute constant.
\end{lem}

\begin{proof}
This is essentially Lemma 3.4 of \cite{Ts:1984}.
\end{proof}

\begin{lem} 
\label{character lemma}
Let $Y=T^{1/\Psi^4}$. For 
$|u| \leq \Psi^2$ we have
\bes
\begin{split}
\frac{1}{T}\int_T^{2T} \exp \bigg(i u \sum_{p \leq Y} \frac{\cos(t \l p)}{p^{1/2}} \Psi^{-1/2} \bigg) dt 
=&\intl_{[0,1]^{\pi(Y)}} \exp \bigg(i u \sum_{p \leq Y} \frac{\cos(2 \pi \theta_p)}{p^{1/2}} \Psi^{-1/2}  \bigg) d\theta \\
&+O\Big(|u|e^{-\Psi^5}\Big),
\end{split}
\ees
where $d\theta=\prod_{p \leq y} d\theta_p$. 
\end{lem}
\begin{proof}
We expand the exponential function to see that
\bes
\begin{split}
\frac{1}{T}\int_T^{2T} \exp \bigg(i u \sum_{p \leq Y}  \frac{\cos(t \l p)}{p^{1/2}} \Psi^{-1/2} \bigg) dt 
=&\sum_{n=0}^{N-1} \frac{(iu \Psi^{-1/2})^n}{n!} \frac1T \int_T^{2T} \!  \bigg( \sum_{p \leq Y} \frac{\cos(t \l p)}{p^{1/2}}  \bigg)^n \, dt\\
&+O\bigg( \frac{|u \Psi^{-1/2}|^N}{N!} \frac1T \int_T^{2T} \!  \bigg| \sum_{p \leq Y} \frac{\cos(t \l p)}{p^{1/2}}  \bigg|^N \, dt \bigg).
\end{split}
\ees
By Lemma \ref{moments lem} we obtain
\bes
\begin{split}
\frac{1}{T}\int_T^{2T} \exp \bigg(i u \sum_{p \leq Y} \frac{\cos(t \l p)}{p^{1/2}} \Psi^{-1/2} \bigg) dt 
=&\sum_{n=0}^{N-1} 
\frac{(iu \Psi^{-1/2})^n}{n!} 
\int\limits_{[0, 1]^{\pi(Y)}} \!  
\bigg(\sum_{p \leq Y} \frac{\cos(2 \pi \theta_p)}{p^{1/2}} \bigg)^n \, d\theta \\
&+O \bigg(\frac{|u|^N (cN)^{N/2}}{N!} \bigg)+O\bigg(T^{-1/2}\sum_{n=1}^{N-1} \frac{|u|^n (c n)^{n/2}}{n!}\bigg).\\
\end{split}
\ees
This equals
\bes
\begin{split}
\intl_{[0,1]^{\pi(Y)}} \exp \bigg(i u \sum_{p \leq Y} \frac{\cos(2 \pi \theta_p)}{p^{1/2}} \Psi^{-1/2}  \bigg) d\theta
&+O \bigg(\frac{|u|^N (cN)^{N/2}}{N!} \bigg)+O\bigg(T^{-1/2}\sum_{n=1}^{N-1} \frac{|u|^n (c n)^{n/2}}{n!}\bigg).
\end{split} 
\ees

We now take $N=2\lfloor \Psi^5 \rfloor$ so that $Y \leq T^{1/N}$, and note that $|u| \leq \Psi^2$. 
We then find that
$$
\frac{|u|^N (cN)^{N/2}}{N!} \ll |u| \frac{\Psi^{2N-2} (cN)^{N/2}}{N!}\ll |u| (c N^{-1}\Psi^4 )^{N/2} \ll |u| e^{-\Psi^5},
$$
by Stirling's formula. The other \textit{O}-term is estimated along the same lines. 

\end{proof}

\begin{lem} 
\label{normal lem}
Let $X=T^{1/100}$. For $|u| \leq \Psi^{1/2}/100$, we have
\bes
\begin{split}
\frac{1}{T}\int_T^{2T} \exp \bigg(i u \sum_{p \leq X^3} \frac{\cos(t \l p)}{p^{1/2}}  \Psi^{-1/2} \bigg) dt
=& e^{-u^2/2}
\(1+O\(\frac{ u^2 \l \l \l T}{ \Psi}\)\)\\
 &+O \bigg( |u|  \bigg(\frac{\l \l \l T}{\Psi} \bigg)^{1/2} \bigg).
\end{split}
\ees
\end{lem} 

\begin{proof}
For real $u, v$, we have $|\exp(iu)-\exp(iv)|=|\int_v^u \exp(it) dt| \leq |u-v|$. Hence,
letting $Y=T^{1/\Psi^4}$, we find that
$$
\frac{1}{T}\int_T^{2T}\exp \bigg(i u \sum_{p \leq X^3} \frac{\cos(t \l p)}{p^{1/2}} \Psi^{-1/2} \bigg) dt
$$
equals
\bes
\frac{1}{T}\int_T^{2T} \exp \bigg(i u \sum_{p \leq Y} \frac{\cos(t \l p)}{p^{1/2}} \Psi^{-1/2} \bigg) dt\\
+O \bigg(\frac{|u|}{\Psi^{1/2}} \frac{1}{T}\int_T^{2T} \bigg| \sum_{Y < p \leq X^3} \frac{\cos(t \l p)}{p^{1/2}} \bigg|dt  \bigg).
\ees
To estimate the integral in the error term we apply Cauchy's inequality 
and Montgomery and Vaughan's mean
value theorem for Dirichlet polynomials \cite{MVMVT:1974} to see that
\bes
\begin{split}
\int_T^{2T} \bigg| \sum_{Y < p \leq X^3} \frac{\cos(t \l p)}{p^{1/2}} \bigg| \, dt 
\ll & T \bigg(\sum_{Y < p \leq X^3} \frac1p\bigg)^{1/2}\\
\ll & T \bigg(\l \frac{3 \l X}{\l Y} \bigg)^{1/2} \\
\ll & T (\l \l \l T)^{1/2}.
\end{split}
\ees
Thus, by this and Lemma \ref{character lemma} it follows that
\be \label{x to y}
\begin{split}
\frac{1}{T}\int_T^{2T} \exp \bigg(  i u \sum_{p \leq X^3} \frac{\cos(t \l p)}{p^{1/2}} \Psi^{-1/2} \bigg) dt 
=&
\intl_{[0,1]^{\pi(Y)}} \exp \bigg(i u \sum_{p \leq Y} \frac{\cos(2 \pi \theta_p)}{p^{1/2}} \Psi^{-1/2}  \bigg) d\theta\\
 & +O\Big(|u|e^{-\Psi^5}\Big)
+O \bigg(|u| \bigg(\frac{\l \l \l T}{\Psi} \bigg)^{1/2} \bigg).
\end{split}
\ee

Next, observe that
\be \label{cf 1}
\begin{split}
\intl_{[0,1]^{\pi(Y)}} \exp \bigg(i u \sum_{p \leq Y} \frac{\cos(2 \pi \theta_p)}{p^{1/2}}  \Psi^{-1/2}  \bigg) d\theta 
=&\prod_{p \leq Y} \int_0^1 \exp\Big(i u \frac{\cos(2 \pi \theta_p)}{p^{1/2}} \Psi^{-1/2} \Big) d\theta_p\\
=&\prod_{p \leq Y} J_0\(\frac{ u}{(p\Psi)^{1/2}}\),
\end{split}
\ee
where $J_0(z)=\int_0^1 e^{iz \cos(2 \pi \theta)} d\theta$ is the zero$th$ Bessel function of the first kind. This function
also has the series expansion
$$
J_0(z)= \sum_{n=0}^{\infty} (-1)^n \frac{(\tfrac12 z)^{2n}}{(n!)^2}.
$$
Consequently, for $|z| \leq 1$ we have
$$
J_0(2z)=e^{-z^2}(1+O(|z|^4)).
$$
It follows that
\be \label{cf 2}
\prod_{p \leq Y} J_0\Big(\frac{ u}{(p\Psi)^{1/2}}\Big)=\exp \bigg(-\frac{u^2}{4\Psi} \sum_{p \leq Y} \frac1p   \bigg)
\prod_{p \leq Y}\Big(1+O\Big(\frac{ u^4}{p^2 \Psi^2}\Big)\Big).
\ee

A simple calculation shows that $\prod_{p \leq Y}(1+O(u^4/(p^2 \Psi^2)))=1+O(u^4/\Psi^2)$. Next, note that
$$
\frac{1}{\Psi} \sum_{p \leq Y} \frac1p=2\frac{\l \l Y+O(1)}{ \l \l T+O(1)}=2(1+O(\l \l \l T/\l \l T)).
$$
Thus, for $|u|\leq \Psi^{1/2}/100$, we have
\bes 
\begin{split}
\exp \bigg(-\frac{ u^2}{4\Psi} \sum_{p \leq Y} \frac1p   \bigg)
\prod_{p \leq Y} \Big(1+O\Big(\frac{ u^4}{p^2 \Psi^2}\Big)\Big)
=\exp( -u^2/2 )
\Big(1+O\Big(\frac{ u^2 \l \l \l T}{ \l \l T}\Big)\Big).
\end{split}
\ees
Therefore, by this, \eqref{x to y}, \eqref{cf 1}, and \eqref{cf 2}
the result follows.
\end{proof}

\section{An Approximate Formula for $\l |\z(\tfrac12+it)|$}
In \cite{Se:1946} Selberg proves an explicit formula for $S(t) =\pi^{-1}\arg \z(\tfrac12+it)$
in terms of a Dirichlet polynomial supported on prime numbers. The purpose of this section is to prove an analogous formula
for $\l |\z(\tfrac12+it)|$.

For any real number
$t$, let
$$
\eta_t=\min_{\g} |t-\g|.
$$
For $x \geq 2$ define
\bes
w_x(n)=
\begin{cases} 
  1 & \mbox{if $n \leq x,$} \\
  \frac{\l^2 (x^3/n)-2\l^2 (x^2/n)}{2 \l^2 x} & \mbox{if $x < n \leq x^2,$} \\
  \frac{\l^2 (x^3/n)}{2 \log^2 x} & \mbox{if $x^2 < n \leq x^3,$} \\
  0 & \mbox{if $n>x^3$}.
\end{cases}
\ees
Next, write $\l^+ x$ for the positive part of the logarithm, that is
$\l^+ x=\l x$ if $x > 1$ and $\l^+ x=0$ for $0 < x \leq 1$. Now, let
\be \label{definition of F}
\begin{split}
F(t; X)=&\Big( \frac{X^{(\frac12-\s_{X,t})/2}}{\l X}+(\s_{X,t}-\tfrac12) \Big) 
 \Big((\s_{X,t}-\tfrac12)\l X+\l^+ \frac{1}{\eta_t \l X}\Big) 
\end{split}
\ee 
where $\s_{X,t}$ is defined in \eqref{sxt def}.
Also, let
\be \label{definition of es}
E_1(t; X)
=\bigg|\sum_{n \leq X^3}
 \frac{\L(n) }{n^{\s_{X,t}+it}}w_X(n) \bigg|. 
\ee

We now cite
\begin{lem} 
\label{tsang dir form}
For  $T < t \leq 2T$ and $2 \leq X \leq T^{1/100}$ we have
\bes
\begin{split}
   \l |\z(\tfrac12+it)|=&\sum_{n \leq X^3}\frac{\L(n) \cos(t \l n)}{n^{\s_{X,t}} \l n}w_X(n)+
   O\Big( F(t; X)\Big(E_1(t;X)+\log T\Big)\Big).
\end{split}
\ees
Additionally, under the same hypotheses
\[
\l |\z(\sigma_{X,t}+it)|=\sum_{n \leq X^3}\frac{\L(n) \cos(t \l n)}{n^{\s_{X,t}} \l n}w_X(n)+
   O\bigg( \frac{X^{(\frac12-\s_{X,t})/2}}{\log X}\Big(E_1(t;X) +\log T\Big)\bigg).
\]
\end{lem} 
\begin{proof}
The first statement is proved in K. M. Tsang's PhD thesis (see Theorem 5.2 of \cite{Ts:1984}). 
The second formula is due to A. Selberg  (see equation (4.9) of \cite{Se:1946}).
\end{proof}

\begin{lem} 
\label{approx form}
For $T < t \leq 2T$ and $2 \leq X \leq T^{1/100}$ we have
\bes
\begin{split}
\l |\z(\tfrac12+it)|=& \sum_{p \leq X^3}\frac{ \cos(t \l p)}{p^{1/2}}+O\Big(F(t; X) \log T+E_2(t; X)+E_3(t; X)\Big)\\
&+O \bigg(F(t; X)X^{\s_{X, t}-\frac12}\int_{1/2}^{\infty} X^{1/2-u}  \bigg|
\sum_{n \leq X^3}
 \frac{\l p \l pX }{p^{u+it}}w_X(p) \bigg|du \bigg),
\end{split}
\ees
where  $F(t; T)$ is defined in \eqref{definition of F} and
$$
E_2(t; X)= \bigg|\sum_{p \leq X^3}\frac{(1-w_X(p))}{p^{1/2+it} } \bigg| \qquad \mbox{and} \qquad
E_3(t; X)= \bigg|\sum_{p \leq X^{3/2}}\frac{ w_X(p^2)}{p^{1+2it} } \bigg|.
$$
Additionally, under the same hypotheses 
\bes
\begin{split}
\l |\z(\sigma_{X,t}+it)|=& \sum_{p \leq X^3}\frac{ \cos(t \l p)}{p^{1/2}}+
O\bigg( \bigg(\frac{X^{(\frac12-\s_{X,t})/2}}{\log X}+(\s_{X,t}-\tfrac12) \bigg)\log T+E_2(t; X)+E_3(t; X)\bigg)\\
&+O \bigg(\bigg(\frac{X^{(\s_{X,t}-\frac12)/2}}{\log X}+(\s_{X,t}-\tfrac12)
X^{\s_{X, t}-\frac12} \bigg)\int_{1/2}^{\infty} X^{1/2-u}  \bigg|
\sum_{n \leq X^3}
 \frac{\l p \l pX }{p^{u+it}}w_X(p) \bigg|du \bigg).
\end{split}
\ees
\end{lem}

 \begin{proof}
We prove only the first assertion. The second statement follows from essentially the same argument.

From Lemma \ref{tsang dir form} we have that
$$
\l |\z(\tfrac12+it)|=\sum_{n \leq X^3}\frac{\L(n) \cos(t \l n)}{n^{\s_{X,t}} \l n}w_X(n)+O\Big( F(t; X)\Big(E_1(t;X)+\log T\Big)\Big)
$$
We now split the sum into 
a sum over primes, 
a sum over squares of primes,
and a sum over the higher prime powers. 
In the sum over primes we replace
the weight $w_X(n)$ with $1$
and $\s_{X, t}$ with $1/2$.
For the sum over squares of primes
we replace $\s_{X, t}$ with $1/2$, and
the sum over the higher prime powers is estimated trivially.
We also use the inequality $|\Re z|\leq |z|$.
In this way we find that
\be \label{formula}
\begin{split}
 \sum_{n \leq X^3}  \frac{\L(n) \cos(t \l n)}{n^{\s_{X, t}} \l n}  w_X(n)
=&\sum_{p \leq X^3}\frac{ \cos(t \l p)}{p^{1/2}} 
+ O\bigg( \bigg|\sum_{p \leq X^3}\frac{ (1-w_X(p))}{p^{1/2+it}} \bigg|\bigg)\\
&+  O\bigg( \bigg|\sum_{p \leq X^3}w_X(p) p^{-it}(p^{-\s_{X,t}}-p^{-1/2}) \bigg| \bigg)
+O\bigg( \bigg|\sum_{p \leq X^{3/2}}\frac{w_X(p^2) }{p^{1+2it}} \bigg|\bigg) \\ 
&+ O\bigg(  \bigg|\sum_{p \leq X^{3/2}}w_X(p^2) p^{-2it}(p^{-2\s_{X,t}}-p^{-1}) \bigg|\bigg)
+O\bigg(\sum_{\substack{p^r \leq X^3 \\ r >2 }} \frac{1}{r p^{r/2}}\bigg).
\end{split}
\ee
The first $O$-term is $\ll E_2(t;X)$ and the third $O$-term is $\ll E_3(t;X)$. 
Next, observe that
$$
\sum_{\substack{p^r \leq X^3 \\ r >2 }} \frac{1}{r p^{r/2}} \ll 1 \ll (\s_{X, t} -\tfrac12) \l T,
$$
so the last error term in \eqref{formula} is bound by $F(t;X)\log T$. To bound
the fourth $O$-term note that
\bes 
\begin{split}
\sum_{p \leq X^{3/2}}w_X(p^2) p^{-2it}(p^{-2\s_{X,t}}-p^{-1})
\ll & \sum_{p \leq X^{3/2}}\frac{1-p^{1-2\s_{X,t}}}{p} \\
\ll & (\s_{X, t}-\tfrac12) \sum_{p \leq X^{3/2}} \frac{\l p}{p} \\
\ll & (\s_{X, t}-\tfrac12) \l T,
\end{split}
\ees
which again is $\ll F(t;X) \log T$.

To complete the proof there are two things to show.
First, that the second error term in \eqref{formula}
is bounded by
\be \label{error bd}
F(t; X) \, X^{\s_{X, t}-1/2} \int_{1/2}^{\infty} X^{1/2-u} \bigg|
\sum_{p \leq X^3}
 \frac{\l p \l pX }{p^{u+it}}w_X(p)\bigg|du.
\ee
The second thing to show is that $F(t;X)E_1(t;X)$ is bounded
by \eqref{error bd} plus $F(t;X)\log T$. 

To begin, observe that
\bes
\begin{split}
\sum_{p \leq X^3}\frac{w_X(p) }{p^{it}} (p^{-\s_{X,t}}-p^{-1/2}) 
= & \int_{1/2}^{\s_{X, t}} \sum_{p \leq X^3}\frac{\l p}{p^{u+it}}w_X(p) du \\
\ll & (\s_{X, t}-\tfrac12) \bigg| \sum_{p \leq X^3}\frac{\l p }{p^{\s^*+it}} w_X(p) \bigg|,
\end{split}
\ees
where $\s^*=\s^*(t)$ lies between $1/2$ and $\s_{X,t}$. Next, we see that if
$1/2 \leq \s \leq \s_{X,t}$, then
\be \label{e1 bd}
\begin{split}
 \bigg| \sum_{p \leq X^3}\frac{\l p }{p^{\s+it}} w_X(p) \bigg|=&\bigg|X^{\s-1/2} 
\int_{\s}^{\infty} X^{1/2-u}  \sum_{p \leq X^3}\frac{ \l p \l Xp}{p^{u+it}} w_X(p) du \bigg| \\
\leq & X^{\s_{X,t}-1/2} 
\int_{1/2}^{\infty} X^{1/2-u} \bigg| \sum_{p \leq X^3}\frac{ \l p \l Xp}{p^{u+it}}w_X(p) \bigg| du. 
\end{split}
\ee
Combining these two estimates, we see that
\bes
\begin{split}
\sum_{p \leq X^3}\frac{w_X(p)}{p^{it}}(p^{-\s_{X,t}}- p^{-1/2}) 
\ll  (\s_{X, t} -\tfrac12)
X^{\s_{X,t}-1/2} \int_{1/2}^{\infty} X^{1/2-u}  \bigg| \sum_{p \leq X^3}\frac{w_X(p) \l p \l Xp}{p^{u+it}}  \bigg| du.
\end{split}
\ees

Finally, to complete the proof we bound $F(t;X)E_1(t;X)$ by \eqref{error bd} plus $F(t;X)\log T$.
To do this observe that
\bes
\begin{split}
F(t; X)\bigg|\sum_{n \leq X^3} \frac{\L(n) }{n^{\s_{X,t}+it}}w_X(n)\bigg| \leq &
F(t; X)\bigg|\sum_{p \leq X^3} \frac{\l p}{p^{\s_{X,t}+it}}w_X(p)\bigg|\\
 &+F(t; X)\bigg|\sum_{p \leq X^{3/2}} \frac{\l p}{p^{2\s_{X,t}+i2t}}w_X(p^2)\bigg| +O(F(t; X)).
\end{split}
\ees
Trivially, the second and third terms on the right-hand side above are $\ll F(t; X)\l T$.
Finally, by \eqref{e1 bd} we have
$$
 \bigg|\sum_{p \leq X^3} \frac{\l p}{p^{\s_{X,t}+it}}w_X(p) \bigg| 
\ll X^{\s_{X,t}-1/2} \int_{1/2}^{\infty} X^{1/2-u} 
 \bigg| \sum_{p \leq X^3}\frac{w_X(p) \l p \l Xp}{p^{u+it}}  \bigg| du.
$$
\end{proof}

We now cite an inequality due to B. Hough.
This enables us to establish an unconditional upper bound on the
distribution function of $\log |\zeta(\tfrac12+ig)| \Psi^{-1/2}$.
\begin{lem} \label{Hough}
Suppose $t >10$. For $t \neq \g$ we have
\begin{equation} \notag
\begin{split}
\log |\zeta(\tfrac12+it)|\le \log |\zeta(\sigma_{X,t}+it)|+\tfrac12 (\s_{X,t}-\tfrac12)\log t+O(1).
\end{split}
\end{equation}
\end{lem}
\begin{proof}
See Proposition 4.1 of
\cite{Hou:2011}. We have applied Stirling's formula to the gamma function term.
\end{proof}
Hough's inequality is similar to one of Soundararajan \cite{So:2009} (see the main proposition). Crucially, the bound here does not depend upon the truth of the Riemann hypothesis.

\section{Controlling the Error Term 
in the Approximate Formula}

In the error term in the approximate formula for $\l |\z(\tfrac12+it)|$, the term 
$\l^+ (1/(\eta_t \l X))$ will be quite large when $t$ is near an ordinate of a 
zero of $\z(s)$. Consequently, the error in approximating $\l |\z(\tfrac12+ig)|$
may be quite large for any given $g$. However, in this section we will
show that Hypothesis~S implies that this can only happen for at most $o(T \l T)$
of the shifted Gram points  with $T<g \leq 2T$. We then show that for the remaining
 $T < g \leq 2T$,
the error term in the approximate formula for $\l |\z(\tfrac12+ig)|$ is
relatively small on average.

We first introduce some notation. We denote by
$g^*$ any $g$ satisfying $\eta_{g} \geq 1 /(\l (|g|+2) \l \l (|g|+3))$.
We denote all other $g$ by $g_*$.

\begin{lem} \label{zeros lem}
Assume Hypothesis~S. Then as $T \rightarrow \infty$
$$
\sum_{T < g_{*} \leq 2T} 1 =o(T \l T).
$$
\end{lem}
\begin{proof}
The proof is by contradiction. Suppose for some
integer $M\geq 2$ there is a sequence $\{ T_n \}$
such that $T_n \rightarrow \infty$ as $n \rightarrow \infty$,
and that for each $n$
\be \label{pos dens}
\frac{1}{N_g(T_n, 2T_n)}\sum_{T_n < g_{*} \leq 2T_n } 1 \geq \frac{1}{M}.
\ee
Let
$$
A_{g_{*}}=A_{g_{\ell}}=\{g_{\ell}, \ldots, g_{{\ell}+M}\}.
$$
By Lemma \ref{simple lem} if $T_n < g_{*}' < g_{*} \leq 2T_n$ and $A_{g_{*}} \cap A_{g_{*}'} \neq \emptyset$
then there is an $m \in \N$ with $1 \leq m \leq M$ such that
$$
\Big|(g_{*}-g_{*}')\frac{\l T_n}{2\pi}-m\Big|\leq \frac{Cm}{\l T_n}=\varepsilon_1,
$$
where $C$ is an absolute constant. 
Thus, for all sufficiently large $n$,
\be \label{key ineq}
\sum_{m=1}^{M} \sum_{T_n < g_{*}, g_{*}' \leq 2T_n } \mathbf 1_{[m-\varepsilon_1, m+\varepsilon_1] }\Big((g_{*}-g_{*}') 
\frac{\l T_n}{2\pi} \Big)
\geq \sum_{\substack{T_n < g_{*}, g_{*}' \leq 2T_n\\ g_{*}'< g_{*}, \, A_{g_{*}} \cap A_{g_{*}'} \neq \emptyset}} 1.
\ee

To obtain a lower bound for the sum on the right-hand side, we begin
by noting that by inclusion-exclusion,
$$
\Big| \bigcup_{T_n < g_{*} \leq 2T_n} A_{g_{*}} \Big|-\sum_{T_n < g_{*} \leq 2T_n} |A_{g_{*}}|
+\sum_{\substack{T_n < g_{*}, g_{*}' \leq 2T_n\\ g_{*}'< g_{*}}} |A_{g_{*}} \cap A_{g_{*}'}| \, \geq 0.
$$
We have $|A_{g_{*}}|=M+1$, $|A_{g_{*}} \cap A_{g_{*}'}| \leq M$ for $g_{*} \neq g_{*}'$, 
and $| \cup_{T_n < g_{*} \leq 2T_n} A_{g_{*}} | \leq N_g(T_n, 2T_n)+M$.
Thus, 
$$
M \sum_{\substack{T_n < g_{*}, g_{*}' \leq 2T_n\\ g_{*}'< g_{*}, \, A_{g_{*}} \cap A_{g_{*}'} \neq \emptyset}} 1 
\geq \sum_{\substack{T_n < g_{*}, g_{*}' \leq 2T_n\\ g_{*}'< g_{*}}} |A_{g_{*}} \cap A_{g_{*}'}|
\geq(M+1)\sum_{T_n < g_{*} \leq 2T_n} 1 -N_g(T_n, 2T_n)-M.
$$
From this and \eqref{pos dens} we obtain
$$
 \sum_{\substack{T_n < g_{*},g_{*}' \leq 2T_n\\ g_{*}'< g_{*}, \, A_{g_{*}} \cap A_{g_{*}'} \neq \emptyset}} 1 
\geq \frac{1}{M^2}N_g(T_n, 2T_n)-1.
$$
Combining this with \eqref{key ineq} we now see that 
\be \label{shift bd 2}
\sum_{m=1}^{M} \sum_{T_n < g_{*},g_{*}' \leq 2T_n } 
\mathbf 1_{[m-\varepsilon_1, m+\varepsilon_1] }\Big((g_{*}-g_{*}') \frac{\l T_n}{2\pi} \Big)
\geq \frac{1}{M^2}N_g(T_n, 2T_n)-1.
\ee

By the definition of $g_*$ we know that there is an ordinate of a zero of $\z(s)$
$\g$ so that
$|g_*-\g| \leq 1/ (\l T_n \l \l T_n)=\varepsilon_2$ for $g_* > T_n$. Hence,
for $T_n<g_*, g_*' \leq 2T_n$
$$
 \bigg|(g_*-g_*') \frac{\l T_n}{2\pi}-(\g-\g')\frac{\l T_n}{2\pi} \bigg| \leq \frac{1}{ \pi \l \l T_n}.
$$
Now let $\varepsilon_3=\varepsilon_1+1/ (\pi \l \l T_n)$. We have
\be \label{gg bd 1}
\begin{split}
\sum_{T_n<g_*, g_*' \leq 2T_n} 
\mathbf 1_{[m-\varepsilon_1, m+\varepsilon_1]}
\Big((g_*-g_*') \frac{\l T_n}{2\pi} \Big) 
 \leq 
\sum_{T_n-\varepsilon_2< \g, \g' \leq 2T_n+\varepsilon_2} 
\mathbf 1_{[m-\varepsilon_3, m+\varepsilon_3]}\Big((\g-\g') \frac{\l T_n}{2\pi} \Big).
\end{split}
\ee
Note that
\be \label{gg bd 2}
\begin{split}
\sum_{T_n-\varepsilon_2< \g, \g' \leq 2T_n+\varepsilon_2} 
\mathbf 1_{[m-\varepsilon_3, m+\varepsilon_3]}\Big((\g-\g') \frac{\l T_n}{2\pi} \Big)  
=& \sum_{T_n< \g, \g' \leq 2T_n} 
\mathbf 1_{[m-\varepsilon_3, m+\varepsilon_3]}\Big((\g-\g') \frac{\l T_n}{2\pi} \Big)\\
&+ O(\l^2 T_n),
\end{split}
\ee
since $N(t+1)-N(t)=O(\l (|t|+2))$ (see \cite{T:1986} Chapter IX).
Now let $\e>0$. Then for all $n$ sufficiently large, we have
\be \label{gg bd 3}
\begin{split}
\sum_{T_n< \g, \g' \leq 2T_n} 
\mathbf 1_{[m-\varepsilon_3, m+\varepsilon_3]}\Big((\g-\g') \frac{\l T_n}{2\pi} \Big) 
\leq 
\sum_{T_n< \g, \g' \leq 2T_n} 
\mathbf 1_{[m-\e, m+\e]}\Big((\g-\g') \frac{\l T_n}{2\pi} \Big).
\end{split}
\ee

By \eqref{gg bd 1}, \eqref{gg bd 2}, and \eqref{gg bd 3} we see that
\bes
\begin{split}
\sum_{T_n<g_*, g_*' \leq 2T_n} 
\mathbf 1_{[m-\varepsilon_1, m+\varepsilon_1]}
\Big((g_*-g_*') \frac{\l T_n}{2\pi} \Big) 
\leq 
\sum_{T_n< \g, \g' \leq 2T_n} 
\mathbf 1_{[m-\e, m+\e]}\Big((\g-\g') \frac{\l T_n}{2\pi} \Big)
+ O(\l^2 T_n).
\end{split}
\ees
Applying this in \eqref{shift bd 2} we see that
$$
\sum_{m=1}^{M} \sum_{T_n< \g, \g' \leq 2T_n} 
\mathbf 1_{[m-\e, m+\e]}\Big((\g-\g') \frac{\l T_n}{2\pi} \Big)
\geq \frac{1}{M^2}N_g(T_n, 2T_n)(1+o(1)).
$$
By Hypothesis~S it follows from this that
\bes
\begin{split}
& \frac{1}{M^2} \leq \lim_{\e \rightarrow 0^+} \bigg( \limsup_{n \rightarrow \infty}
 \sum_{m=1}^{M}\frac{1}{N_g(T_n, 2T_n)}  \sum_{T_n< \g, \g' \leq 2T_n} 
\mathbf 1_{[m-\e, m+\e]}\Big((\g-\g') \frac{\l T_n}{2\pi} \Big)\bigg) \\
&=
 \sum_{m=1}^{M} \lim_{\e \rightarrow 0^+} \bigg( \limsup_{n \rightarrow \infty}
  \frac{ |\{T_n< \g, \g' \leq 2T_n :  
| \frac{(\g-\g')\l T_n}{2\pi}-m |<\epsilon \}|}{N_g(T_n, 2T_n)}\bigg)=0,
\end{split}
\ees
so we have reached a contradiction.
\end{proof}
The following lemma will allow us to show 
that the error term in the approximate formula
for $\l |\z(1/2+ig^*)|$ with $T < g^* \leq 2T$ is
relatively small on average. 
\begin{lem} \label{lem error bd}
 Let $X = T^{1/100}$ and $F(t; X)$ be as defined 
 in \eqref{definition of F}. Then
$$
\sum_{T < \gad \leq 2T} F(\gad; X) \ll   T \l \l \l T
$$
and
$$
\sum_{T < \gad \leq 2T} X^{2\s_{X, \gad}-1}F(\gad; X)^2 \ll  T \frac{(\l \l \l T)^2}{\l T}.
$$
\end{lem}

\begin{proof}
Recall that
\bes
\begin{split}
F(t; X)=&\Big( \frac{X^{(1/2-\s_{X,t})/2}}{\l X}+(\s_{X,t}-\tfrac12) \Big)
 \Big((\s_{X,t}-\tfrac12)\l X+\l^+ \frac{1}{\eta_t \l X}\Big).
\end{split}
\ees
By definition $\eta_{g^*} \gg 1/(\l T \l \l T)$ for $T< g^* \leq 2T$ and $\s_{X,t} \geq 1/2+4/\l X$. So that
$$
F(g^*; X) \ll  (\s_{X,\gad}-\tfrac12)^2\l X+(\s_{X,\gad}-\tfrac12)\l \l \l T .
$$
Both assertions of the lemma now follow from Lemma \ref{zero density bd}.
\end{proof}

\section{The Proofs of Theorem \ref{upper bd} and Theorem \ref{thm 2}}
Let
\[
 N_{\ga}^* (T, 2T)=\sum_{\substack{T < \ga \leq 2T \\ g\neq \g}} 1
\]
and
$$
\mathcal F_T(v)=\frac{|\{ T < g \leq 2T,  g\neq \g: \l
|\z(\tfrac12+ig)|\Psi^{-1/2} \leq v\}|}{N_{\ga}^*(T, 2T)}.
$$
Note that $\mathcal F_T(v)$ is a distribution function. It also follows
that the characteristic function corresponding to $\mathcal F_T(v)$ is given by
$$
\phi_T(u)
:=\int_{-\infty}^{\infty} e^{iu v} \, d \mathcal F_T(v)=
\frac{1}{N_{\ga}^* (T, 2T)}
\sum_{\substack{T < \ga \leq 2T \\ g\neq \g}} \exp\Big(i u \l |\z(\tfrac12+i\ga)|\Psi^{-1/2}\Big).
$$
Similarly, the distribution function for $\log 
|\zeta(\sigma_{X,g}+ig)| \Psi^{-1/2}$ at the points
$T < g \le 2T$ has the corresponding characteristic function
\[
\varphi_T(u)=\frac{1}{N_{\ga} (T, 2T)}
\sum_{\substack{T < g \leq 2T}} \exp\Big(i u \l 
|\z(\sigma_{X,g}+ig)|\Psi^{-1/2}\Big).
\]
(By the definition of $\s_{X,t}$ we know that
$\s_{X,g}+ig$ is not a zero of $\zeta(s)$.)

Given
distribution functions $\mathcal F_1, \mathcal F_2, \ldots$ and 
the corresponding characteristic functions $ \phi_1, \phi_2, \ldots$ 
L\'evy's
continuity theorem states, in particular, that 
if $\phi_n$ converges pointwise on $\R$ to 
a function $\phi$ that is continuous at $0$ as $n \rightarrow \infty$
then $\mathcal F_n$ 
converges weakly to a distribution function $\mathcal F$ as $n \rightarrow \infty$.
Moreover, the
characteristic function of $\mathcal F$ is $\phi$.
See 
Billingsley \cite{Bi:1986} Theorem 26.3 or Theorem 3 from Chapter III.2
of Tenenbaum \cite{Ten:1995}. Also, note that an analogue of L\'evy's continuity theorem
holds when one replaces $\{ \mathcal F_n\}$ by $\{\mathcal F_T : T>T_0\}$ and $\{ \phi_n \}$ by
$\{ \phi_T : T>T_0\}$ for constant $T_0$.

We shall prove
\begin{prop} \label{ch.f. sxt}
Let $X=T^{1/100}$.
For $|u| \leq \Psi^{1/2}/100$, we have
\bes
\begin{split}
\frac{1}{N_{\ga} (T, 2T)}\sum_{\substack{T < \ga \leq 2T }} 
\exp\Big(i u \l |\z(\sigma_{X,g}+i\ga)|\Psi^{-1/2} \Big)=&\,e^{-u^2/2}\(1+O\(\frac{u^2 \l \l \l T}{\Psi}\)\)\\
&+O\bigg(  |u|\bigg(\frac{ \log \log \log T}{\Psi}\bigg)^{1/2}\bigg) +O\Big(\frac{1}{\l T}\Big).
\end{split}
\ees
\end{prop}

 Moreover, under the assumption of Hypothesis~S we shall prove
\begin{prop} \label{ch.f.} 
Assume Hypothesis~S. Then for $|u| \leq \Psi^{1/2}/100$, we have
\bes
\begin{split}
\frac{1}{N_{\ga}^* (T, 2T)}\sum_{\substack{T < \ga \leq 2T \\ g\neq \g}} 
\exp\Big(i u \l |\z(\tfrac12+i\ga)|\Psi^{-1/2} \Big)=&\,e^{-u^2/2}\(1+O\(\frac{u^2 \l \l \l T}{\Psi}\)\)\\
&+o(1)+O\( |u| \frac{\l \l \l T}{\Psi^{1/2}} \).
\end{split}
\ees
\end{prop}
Note that $e^{-u^2/2}$ is the characteristic function of a normally distributed
random variable with mean zero and variance one. Hence, Theorem
\ref{thm 2} immediately follows from Proposition \ref{ch.f.} and
 L\'evy's continuity theorem.

We will now deduce Theorem \ref{upper bd} using Proposition \ref{ch.f. sxt} and Lemma \ref{Hough}.

\begin{proof}[Proof of Theorem \ref{upper bd}]
By Lemma \ref{Hough}  there is an absolute constant $C>0$ such that for $T<t \le 2T$
we have for $t \neq \g$ that
\[
\log |\zeta(\tfrac12+it)| \le \log |\zeta(\s_{X,t}+it)|+\tfrac12(\s_{X,t}-\tfrac12)\log T+C.
\]
Writing $\Delta(t)=((\s_{X,t}-\tfrac12)\log T+C)\Psi^{-1/2}$ we see that
\begin{equation} \label{prop sxt 1}
 \sum_{\substack{T < g \leq 2T \\ g \neq \g}}
\mathbf 1_{[\alpha, \infty)} \Big( \l |\z(\tfrac12+ig)|\, \Psi^{-1/2} \Big) 
\le   \sum_{\substack{T < g \leq 2T }}
\mathbf 1_{[\alpha-\Delta(g), \infty)} \Big( \l |\z(\s_{X,g}+ig)|\, \Psi^{-1/2} \Big). 
\end{equation}

Using Lemma \ref{zero density bd} and Chebyshev's inequality we have
for any fixed $\e>0$
\[
\sum_{\substack{T< g \leq 2T \\ \Delta(g) \geq \e}} 1
 \le
\e^{-1} \sum_{T< g \leq 2T} \Delta(g)
 \ll \frac{ T}{\e \Psi^{1/2}} \log T =o(T\l T) .
\]
Thus,
\begin{equation} \label{prop sxt 2}
\sum_{\substack{T < g \leq 2T }}
\mathbf 1_{[\alpha-\Delta(g), \infty)} \Big( \l |\z(\s_{X,g}+ig)|\, \Psi^{-1/2} \Big)
\leq \sum_{\substack{T < g \leq 2T }}
\mathbf 1_{[\alpha-\epsilon, \infty)} \Big( \l |\z(\s_{X,g}+ig)|\, \Psi^{-1/2} \Big)
+o(T\l T).
\end{equation}
Hence, Proposition \ref{ch.f. sxt} and L\'evy's continuity theorem imply that
\[
\frac{1}{N_g(T,2T)}
\sum_{\substack{T < g \leq 2T }}
\mathbf 1_{[\alpha-\epsilon, \infty)} \Big( \l |\z(\s_{X,g}+ig)|\, \Psi^{-1/2} \Big)
=\frac{1}{\sqrt{2\pi}}\int_{a-\epsilon}^{\infty} e^{-x^2/2} \, dx+o(1).
\]
Since $\e$ is arbitrary Theorem \ref{upper bd} now follows from this, \eqref{prop sxt 1} and \eqref{prop sxt 2}.

\end{proof}

\begin{lem} \label{ch.f. dir} Let $X=T^{1/100}$.
For $|u| \leq \Psi^{1/2}/100$, we have
\bes
\begin{split}
\frac{1}{N_g(T, 2T)} \sum_{T < \ga \leq 2T} \exp \bigg(i u \sum_{p \leq X^3}\frac{ \cos(\ga \l p)}{p^{1/2}} \Psi^{-1/2} \bigg)
= & e^{- u^2/2} \bigg(1+O \bigg(\frac{u^2 \l \l \l T}{\Psi} \bigg)  \bigg)\\
&+O \bigg(|u| \bigg(  \frac{\l \l \l T}{\Psi}  \bigg)^{1/2} \bigg)
+O(1/\l T).
\end{split}
\ees
\end{lem}

\begin{proof}
By Stieltjes integration
\bes
\begin{split}
\sum_{T < \ga \leq 2T} \exp \bigg(i u \sum_{p \leq X^3}\frac{ \cos(\ga \l p)}{p^{1/2}} \Psi^{-1/2} \bigg)
= \int_{T}^{2T} \exp \bigg(i u \sum_{p \leq X^3}\frac{ \cos(t \l p)}{p^{1/2}}\Psi^{-1/2} \bigg) d\mu(t),
\end{split}
\ees
where $\mu(t)=\mu_{\phi}(t)=\lfloor (\theta(t)+\phi)/\pi \rfloor=t/(2\pi) \l (t/(2\pi e))+r(t)$
and $r(t) \ll 1$ for $t \geq 10$. The right-hand side of this equals
\be \label{2 ints}
\begin{split}
\int_T^{2T} \exp \bigg( i u \sum_{p \leq X^3}\frac{ \cos(t \l p)}{p^{1/2}} \Psi^{-1/2} \bigg)  d\Big(\frac{t}{2\pi}\l\frac{t}{2\pi e}\Big)
+
\int_{T}^{2T} \exp \bigg(i u \sum_{p \leq X^3}\frac{ \cos(t \l p)}{p^{1/2}}\Psi^{-1/2} \bigg) dr(t).
\end{split}
\ee
The first integral equals
$$
\frac{1}{2\pi}\int_T^{2T} 
\exp \bigg(i u \sum_{p \leq X^3}\frac{ \cos(t \l p)}{p^{1/2}} \Psi^{-1/2} \bigg) \l \frac{t}{2\pi} dt,
$$
which by the Second Mean Value Theorem equals
$$
\frac{1}{2\pi} \l \frac{T}{2\pi} \int_T^{2T} 
\exp \bigg( i u \sum_{p \leq X^3}\frac{ \cos(t \l p)}{p^{1/2}} \Psi^{-1/2} \bigg)  dt+O(T).
$$
By Lemma \ref{normal lem} this equals
\be \label{fin one}
\begin{split}
\frac{T}{2\pi} \l \frac{T}{2\pi} e^{-u^2/2} \bigg(1+O \bigg(\frac{u^2 \l \l \l T}{\Psi} \bigg)  \bigg)
+O \bigg(|u|T \l T  \bigg(\frac{\l \l \l T}{\Psi} \bigg)^{1/2} \bigg)+O(T)
\end{split}
\ee
for $|u| \leq  \Psi^{1/2}/100$. 

It remains to bound the second integral 
in \eqref{2 ints}.
Integrating by parts we see that
\bes
\begin{split}
\int_{T}^{2T} \exp \bigg(i u \sum_{p \leq X^3}\frac{ \cos(t \l p)}{p^{1/2}}  \Psi^{-1/2} \bigg) dr(t) 
\ll 1 +\bigg| \int_{T}^{2T} r(t) 
d  \bigg(\exp \bigg(i u \sum_{p \leq X^3}\frac{ \cos(t \l p)}{p^{1/2}} \Psi^{-1/2}  \bigg)  \bigg)\bigg|.
\end{split}
\ees
Now
\bes
\begin{split}
\int_{T}^{2T} r(t) d  \bigg(\exp \bigg(i u \sum_{p \leq X^3}\frac{ \cos(t \l p)}{p^{1/2}}  \Psi^{-1/2} \bigg) \bigg)
\ll
\frac{|u|}{\Psi^{1/2}} \int_T^{2T}\bigg|  \sum_{p \leq X^3} \frac{ \sin(t \l p) \l p}{p^{1/2}}\bigg|dt .
\end{split}
\ees
Applying Cauchy's inequality and then Montgomery and Vaughan's 
mean value theorem for Dirichlet polynomials \cite{MVMVT:1974}, we see that
\bes
\begin{split}
\int_T^{2T}\Big|  \sum_{p \leq X^3} \frac{ \sin(t \l p) \l p}{p^{1/2}}\Big|dt \leq & T^{1/2}
\bigg( \int_T^{2T}\bigg|  \sum_{p \leq X^3} \frac{ \l p}{p^{1/2+it}}\bigg|^2dt\bigg)^{1/2} \\
 \ll & T^{1/2}(T \l^2 T)^{1/2}= T\l T.
\end{split}
\ees
We now have that
$$
\int_{T}^{2T} \exp \bigg(i u \sum_{p \leq X^3}\frac{ \cos(t \l p)}{p^{1/2}} \Psi^{-1/2} \bigg) dr(t)\ll 1
+ \frac{|u|}{ \Psi^{1/2}} T \l T.
$$
Combining this and \eqref{fin one} in \eqref{2 ints} we obtain
\bes
\begin{split}
\frac{1}{N_g(T, 2T)} \sum_{T < \ga \leq 2T} \exp \bigg(i u \sum_{p \leq X^3}\frac{ \cos(\ga \l p)}{p^{1/2}} \Psi^{-1/2} \bigg)
= & e^{- u^2/2} \bigg(1+O \bigg(\frac{u^2 \l \l \l T}{\Psi} \bigg)  \bigg)\\
&+O \bigg(|u| \bigg(  \frac{\l \l \l T}{\Psi}  \bigg)^{1/2} \bigg)
+O(1/\l T).
\end{split}
\ees
\end{proof}

Observe that Proposition \ref{ch.f. sxt}
follows immediately from Lemma \ref{ch.f. dir} and 

\begin{lem} \label{ch. f.}
Let $X=T^{1/100}$. For $|u| \leq \Psi^{1/2}/100$, we have
\be \label{prop62 3}
\begin{split}
\sum_{T < \gad \leq 2T} \exp\Big(iu \l |\z(\tfrac12+i\gad)| \Psi^{-1/2} \Big)
=&\sum_{T < \gad \leq 2T} \exp\bigg(iu \sum_{p \leq X^3}\frac{ \cos(\gad \l p)}{p^{1/2}}\Psi^{-1/2}\bigg)\\
&+O\bigg(\frac{|u| T \l T \l \l \l T}{ \Psi^{1/2}} \bigg).
\end{split}
\ee
Additionally, under the same hypotheses we have
\bes
\sum_{T < g \leq 2T} \exp\Big(iu \l |\z(\s_{X,g}+ig)| \Psi^{-1/2} \Big)
=\sum_{T < g \leq 2T} \exp\bigg(iu \sum_{p \leq X^3}\frac{ \cos(g \l p)}{p^{1/2}}\Psi^{-1/2}\bigg)
+O\bigg(\frac{|u| T \l T }{ \Psi^{1/2}} \bigg).
\ees
\end{lem}

\begin{proof}
We will omit the proof of the second assertion
as it follows from a similar argument. Note that
in the second statement the
error term is slightly smaller because we can obtain better estimates on
averages of $X^{(\s_{X,g}-\frac12)}$ and $(\s_{X,g}-\tfrac12)$
 as opposed to averages of $F(g^*; X)$. (Compare
Lemma \ref{zero density bd} to Lemma \ref{lem error bd}.)

Note that $|e^{ia}-e^{ib}|=|\int_a^b e^{it} dt|\leq |b-a|$. By this and Lemma \ref{approx form} we have
\be \label{error form}
\begin{split}
& \sum_{T < \gad \leq 2T}   \bigg| \exp\Big(i u \l |\z(\tfrac12+i\gad)| \Psi^{-1/2} \Big) 
-\exp \bigg(iu \sum_{p \leq X^3}\frac{ \cos(\gad \l p)}{p^{1/2}} \Psi^{-1/2} \bigg) \bigg|\\
& \ll  \frac{|u|}{\Psi^{1/2}} \sum_{T < \gad \leq 2T} \bigg( F(\gad; X)X^{\s_{X, \gad}-1/2}  
\int_{1/2}^{\infty} X^{1/2-u}  \bigg|
\sum_{p \leq X^3}
 \frac{\l p \l pX }{p^{u+i\gad}}w_X(p) \bigg|du \\
 &  \qquad \qquad \qquad \qquad \quad
 + F(\gad;X)\log T+E_2(\gad;X)+E_3(\gad;X)\bigg).
\end{split}
\ee
(See \eqref{definition of es} for the definition of $E_1$ and Lemma \ref{approx form}
for the definitions of $E_2, E_3$.)
Observe that $(1-w_X(p)) \ll \l p/ \l X$, so by Cauchy's inequality and \eqref{MV 1},
$$
 \sum_{T < \gad \leq 2T} E_2(\gad;X) \ll (T \l T)^{1/2}  \bigg(\sum_{T < g \leq 2T} E_2(\ga;X)^2 \bigg)^{1/2} \ll T \l T.
$$
Similarly, by \eqref{MV 2}
$$
\sum_{T < \gad \leq 2T} E_3(\gad;X) \ll (T \l T)^{1/2}  \bigg( \sum_{T < \ga \leq 2T} 
E_3(g;X)^2 \bigg)^{1/2} \ll T \l T.
$$
Next, by Lemma \ref{lem error bd} 
$$
 \sum_{T < \gad \leq 2T} F(\gad; X)  \l T \ll T \l T (\l \l \l T).
$$

It remains to bound the first term
in \eqref{error form}.
We begin by applying Cauchy's inequality to see that
\be \label{eqn}
\begin{split}
& \sum_{T < \gad \leq 2T}  F(\gad; X) X^{\s_{X, \gad}-1/2} \int_{1/2}^{\infty} X^{1/2-u} \bigg|
\sum_{p \leq X^3}
 \frac{\l p \l pX }{p^{u+i\gad}}w_X(p)\bigg|du \\
 \leq & \bigg(\sum_{T < \gad \leq 2T}F^2(\gad; X) X^{2\s_{X, \gad}-1} \bigg)^{1/2} 
   \bigg(\sum_{T < \gad \leq 2T} 
  \bigg(\int_{1/2}^{\infty} X^{1/2-u}  \bigg|
\sum_{p \leq X^3}
 \frac{\l p \l pX }{p^{u+i\gad}}w_X(p) \bigg|du\bigg)^2\bigg)^{1/2}.
\end{split}
\ee
By Lemma \ref{lem error bd}, 
\be \label{eqn 1}
\sum_{T < \gad \leq 2T}F^2(\gad; X) X^{2\s_{X, \gad}-1} \ll T\frac{(\l \l \l T)^2}{\l T}.
\ee
Next, apply Cauchy's inequality to the final factor
on the right-hand side of \eqref{eqn} to see that
\bes
\begin{split}
\sum_{T < \gad \leq 2T} 
  \bigg(\int_{1/2}^{\infty} X^{1/2-u} & \bigg|
\sum_{p \leq X^3}
 \frac{\l p \l pX }{p^{u+i\gad}}w_X(p) \bigg|du\bigg)^2 \\
 \leq& \(\int_{1/2}^{\infty} X^{1/2-u} du\)
  \bigg(\int_{1/2}^{\infty} X^{1/2-u}  \sum_{T < \gad \leq 2T} \bigg|
\sum_{p \leq X^3}
 \frac{\l p \l pX }{p^{u+i\gad}}w_X(p) \bigg|^2  du \bigg) \\
 \ll& \int_{1/2}^{\infty} X^{1/2-u} \l^3 X  \sum_{T < \ga \leq 2T}  \bigg|
\sum_{p \leq X^3}
 \frac{\l p \l pX }{p^{u+i\ga}\l^2 X}w_X(p) \bigg|^2  du.
\end{split}
\ees
Noting that for $p \leq X^3$ we have $\l p \l pX/(p^{u}\l^2 X)  \ll \l p/(p^{1/2}\l X)$
for $u \geq 1/2$, we see from
\eqref{MV 1} that the sum over $g$ is $\ll T \l T$ uniformly for $u \geq 1/2$.
Hence, the quantity above is $\ll T \l^3 T$. 
Combining this with \eqref{eqn}
and \eqref{eqn 1} we see that
\bes
\begin{split}
\sum_{T < \gad \leq 2T}  F(\gad; X) X^{\s_{X, \gad}-1/2} \int_{1/2}^{\infty} X^{1/2-u}  \bigg|
\sum_{p \leq X^3}
 \frac{\l p \l pX }{p^{u+i\gad}}  w_X(p) \bigg|  du  \ll  T \l T \l \l \l T.
\end{split}
\ees
\end{proof}

\begin{proof}[Proof of Proposition \ref{ch.f.}]
Since $|e^{it}|=1$ we see that by Lemma \ref{zeros lem} extending the 
range of the sum on the left-hand side of \eqref{prop62 3} to $T< g \le 2T$, $g\neq \g$
produces an error term of size at most $o(T \l T)$.
Similarly, extending the sum on the right-hand side of \eqref{prop62 3} to all of $T<g \le 2T$
gives an error term not exceeding $o(T \l T)$. Also, by Lemma \ref{zeros lem} $N_g^*(T,2T)=N_g(T,2T)(1+o(1))$. 
Hence, by these observations we have for $|u| \le \Psi^{1/2}/100$ that
\bes
\begin{split}
\frac{1}{N_g^*(T,2T)}\sum_{\substack{T < g \leq 2T\\ g\neq \g} }
\exp\Big(iu \l |\z(\tfrac12+ig)| \Psi^{-1/2} \Big)
=&\frac{1}{N_g(T, 2T)}\sum_{T < g \leq 2T} 
\exp\bigg(iu \sum_{p \leq X^3}\frac{ \cos(g \l p)}{p^{1/2}}\Psi^{-1/2}\bigg)\\
&+O\bigg(\frac{|u| \l \l \l T}{ \Psi^{1/2}} \bigg)
+o(1).
\end{split}
\ees
Applying Lemma \ref{ch.f. dir} completes the proof.
\end{proof}

\subsection{Acknowledgments}
The majority of this article is part of the author's PhD thesis, which was supervised
by Prof. Steven Gonek. I would like to thank Prof. Gonek for suggesting
this problem and also for his guidance and encouragement. The
exposition of this article has greatly benefited from his advice and suggestions.
I am also very grateful
to the anonymous referee for giving 
helpful suggestions and remarks, which significantly improved this article. In particular, for pointing
out the unconditional inequality of Bob Hough and suggesting
Theorem \ref{upper bd}. This theorem is new to this version.


\begin{thebibliography}{0}

\bibitem{Ba:2012}   Banks, W., V. Castillo-Garate, L. Fontana, and C. Morpurgo,
``Self-intersections of the Riemann zeta-function on the critical line." \textit{Journal of Mathematical Analysis and Applications} \textbf{406}, no. 2 (2013): 475-481.

\bibitem{Bi:1986}   Billingsley, P. \textit{Probability and Measure}, 2nd ed.
New York: John Wiley, 1986.


\bibitem{BH2:1987} Bombieri, E. and D. A. Hejhal.
``On the zeros of Epstein zeta functions."
\textit{Comptes Rendus de l'Acad\'emie des 
Sciences Paris S\'erie I Math\'ematiques} \textbf{304},
no. 9 (1987): 213-217.

\bibitem{BH:1995} Bombieri, E. and D. A. Hejhal.
``On the distribution of zeros of linear combinations
of Euler products". \textit{Duke Mathematical Journal} 
\textbf{80},  no. 3 (1995):
821-862.



\bibitem{B:1948} Borchsenius, V. and B. Jessen. 
``Mean Motions and Values of the Riemann Zeta Function." 
\textit{Acta Mathematica} \textbf{80}, no. 1 (1948): 97-166.

\bibitem{F:2013} Farmer, D. W., S. M. Gonek,
and Y. Lee.
``Pair Correlation of the zeros of the derivative of the Riemann $\xi$-function."
(2013): preprint.

 
 
\bibitem{H:1987}  Hejhal, D. A.
``On the distribution of $\l |\z'(\tfrac12+it)|$." 
\textit{Number theory, trace formulas and discrete groups} (Oslo, 1987), 343-370. Boston: Academic Press, 1989.

\bibitem{Hou:2011} Hough, B. ``The Distribution of the logarithm of orthogonal and symplectic
$L$-functions." (2011): preprint arXiv:1109.1783.
 
\bibitem{L:1975}   Levinson, N. ``Almost all roots of
$\zeta(s)=a$ are arbitrarily close to $\sigma=1/2$." 
\textit{Proceeding of the Natural 
Academy of Sciences of the United States of America}  \textbf{72}, (1975):1322-1324.



\bibitem{MVMVT:1974}  Montgomery, H. L. and  R. C. Vaughan. ``Hilbert's
inequality." \textit{ Journal of the London Mathematical Society (2)} \textbf{8}, (1974):73-82.


\bibitem{Se:1946}  Selberg, A. ``Contributions to the theory of the Riemann zeta-function." \textit{Archiv 
for Mathematik og Naturvidenskab} \textbf{48}, no. 5 (1946): 89-155.

\bibitem{S:1992} Selberg, A. ``Old and new conjectures and results about a class of Dirichlet series." In 
\textit{Proceedings of the Amalfi Conference on Analytic Number Theory (Maiori 1989)},367-385. Salerno: Universit\`a di Salerno 1992.

\bibitem{So:2009}  Soundararajan, K., 
``Moments of the Riemann zeta function",  
{\em Ann. of Math.} (2) \textbf{170} (2009), no. 2, 981-993. 

\bibitem{Ten:1995} Tenenbaum, G. \textit{Introduction to analytic
and probabilistic number theory}, Cambridge: Cambridge University Press, 1995.

\bibitem{T:1986}  Titchmarsh, E. C. \textit{The theory of the Riemann zeta-function}, 2nd ed. Revised by D.R. 
Heath-Brown. Oxford: Oxford University Press, 1986. 


\bibitem{Ts:1984}  Tsang, K. M.``The distribution of the values of the Riemann zeta-function." 
PhD diss., Princeton Univ., Princeton, 1984.

\end{thebibliography}
\end{document}